\documentclass[12pt,reqno]{amsart}
\usepackage{amsmath,amssymb,extarrows}
\usepackage{url}
\usepackage{tikz,enumerate}
\usepackage{diagbox}
\usepackage{appendix}
\usepackage{epic}

\usepackage{float} 

\usepackage{cite}

\usepackage{hyperref}
\usepackage{array}

\usepackage{booktabs}

\allowdisplaybreaks[4]

\newtheorem{theorem}{Theorem}
\newtheorem{corollary}[theorem]{Corollary}

\theoremstyle{definition}

\theoremstyle{remark}
\newtheorem{rem}{Remark}
\numberwithin{equation}{section}
\numberwithin{theorem}{section}
\numberwithin{defn}{section}

\newcommand*\diff{\mathop{}\!\mathrm{d}}

\begin{document}
\title[Multi-sum Rogers-Ramanujan type identities]
 {Multi-sum Rogers-Ramanujan type identities}

\author{Zhineng Cao and Liuquan Wang}
\address{School of Mathematics and Statistics, Wuhan University, Wuhan 430072, Hubei, People's Republic of China}
\email{zhncao@whu.edu.cn}
\address{School of Mathematics and Statistics, Wuhan University, Wuhan 430072, Hubei, People's Republic of China}
\email{wanglq@whu.edu.cn;mathlqwang@163.com}




\begin{abstract}
We use an integral method to establish a number of Rogers-Ramanujan type identities involving double and triple sums. The key step for proving such identities is to find some infinite products whose integrals over suitable contours are still infinite products. The method used here is motivated by Rosengren's proof of the Kanade-Russell identities.
\end{abstract}

\maketitle

\section{Introduction}\label{sec-intro}
The famous Rogers-Ramanujan identities assert that
\begin{align}\label{RR}
\sum_{n=0}^\infty \frac{q^{n^2}}{(q;q)_n}=\frac{1}{(q,q^4;q^5)_\infty}, \quad \sum_{n=0}^\infty \frac{q^{n(n+1)}}{(q;q)_n}=\frac{1}{(q^2,q^3;q^5)_\infty}.
\end{align}
Here and throughout this paper, we assume that $|q|<1$ for convergence and use the standard $q$-series notation
\begin{align}
(a;q)_0:=1, \quad (a;q)_n:=\prod\limits_{k=0}^{n-1}(1-aq^k), \quad (a;q)_\infty :=\prod\limits_{k=0}^\infty (1-aq^k),  \\
(a_1,\cdots,a_m;q)_n:=(a_1;q)_n\cdots (a_m;q)_n, \quad n\in \mathbb{N}\cup \{\infty\}.
\end{align}

These two sum-product identities have fascinating combinatorial interpretations, and they stimulate a number of researches on finding similar identities. One of the famous work on this direction is Slater's list \cite{Slater}, which contains 130 of such identities such as
\begin{align}
    \sum_{n=0}^\infty \frac{q^{2n^2}}{(q;q)_{2n}}&=\frac{1}{(q^2,q^3,q^4,q^5,q^{11},q^{12},q^{13},q^{14};q^{16})_\infty}, \\
    \sum_{n=0}^\infty \frac{q^{2n(n+1)}}{(q;q)_{2n+1}}&= \frac{1}{(q,q^4,q^6,q^7,q^9,q^{10},q^{12},q^{15};q^{16})_\infty}.
\end{align}
Identities similar to \eqref{RR} are called as Rogers-Ramanujan type identities.

It is natural to consider multi-sum Rogers-Ramanujan type identities. For example, the Andrews-Gordon identity (see \cite{Andrews1974,Gordon1961}), which is a generalization of \eqref{RR}, states that for positive integer $k>1$ and $1\leq i \leq k$,
\begin{align}
&\sum_{n_{k-1}\geq n_{k-2}\geq \cdots \geq n_1\geq 0} \frac{q^{n_1^2+n_2^2+\cdots+n_{k-1}^2+n_i+n_{i+1}+\cdots +n_{k-1}}}{(q;q)_{n_{k-1}-n_{k-2}}(q;q)_{n_{k-2}-n_{k-3}}\cdots (q;q)_{n_2-n_1} (q;q)_{n_1}}  \nonumber \\ &=\frac{(q^i,q^{2k+1-i},q^{2k+1};q^{2k+1})_\infty}{(q;q)_\infty}. \label{AG}
\end{align}
Bressoud \cite{Bressoud1980} provided an even modulus analog of this identity.
In a series of works (see e.g. \cite{Lepowsky-Wilson,Lepowsky-Wilson-1985}), Lepowsky and Wilson developed Lie theoretic approach to establish  Rogers-Ramanujan type identities. In particular, they showed that the Rogers-Ramanujan identities, the Andrews-Gordon identity and Bressoud's identity are closely related to the affine Kac-Moody  Lie algebra $A_1^{(1)}$. This motivates people to find similar identities by studying other Lie algebras.  See the books \cite{Lost2,Sills-book} for more historical background.

In recent years, Kanade and Russell \cite{KR-2019} searched for Rogers-Ramanujan type identities related to level 2 characters of the affine Lie algebra $A_9^{(2)}$, and they conjectured a number of such identities. Let
\begin{align}
F(u,v,w)&:=\sum_{i,j,k\geq 0} \frac{(-1)^kq^{3k(k-1)+(i+2j+3k)(i+2j+3k-1)}u^iv^jw^k}{(q;q)_i(q^4;q^4)_j(q^6;q^6)_k}, \\
G(u,v,w)&:=\sum_{i,j,k\geq 0}\frac{q^{(i+2j+3k)(i+2j+3k-1)/2+j^2}u^iv^jw^k}{(q;q)_i(q^2;q^2)_j(q^3;q^3)_k}.
\end{align}
Some of their conjectural identities are
\begin{align}
F(q,1,q^3)&=\frac{(q^3;q^{12})_\infty}{(q,q^2;q^4)_\infty}, \label{KR-conj-1} \\
F(q,q,q^6)&=\frac{1}{(q^3;q^4)_\infty (q,q^8;q^{12})_\infty}, \label{KR-conj-2} \\
G(q,q^2,q^4)&=\frac{1}{(q;q^3)_\infty (q^3,q^6,q^{11};q^{12})_\infty}, \label{KR-conj-3} \\
G(q^2,q^4,q^5)&=\frac{1}{(q^2;q^3)_\infty (q^3,q^6,q^7;q^{12})_\infty}.  \label{KR-conj-4}
\end{align}
Five of their conjectural identities on $F(u,v,w)$ as well as the identities \eqref{KR-conj-3} and \eqref{KR-conj-4} on $G(u,v,w)$ were confirmed by Bringmann, Jennings-Shaffer and Mahlburg \cite{BSM}. Later, using an integral method, Rosengren \cite{Rosengren} gave proofs to all of the nine conjectural identities on $F(u,v,w)$.

Since there are numerous Rogers-Ramanujan type identities in the literature and some of them have similar shapes, it is more convenient to group some of them together. Following the notion in \cite{Wang}, for a fixed $k$, we shall call an identity of the following shape: finite sum of
\begin{align}\label{type-defn}
\sum_{(i_1,\cdots,i_k)\in S}\frac{(-1)^{t(i_1,\cdots,i_k)}q^{Q(i_1,\cdots,i_k)}}{(q^{n_1};q^{n_1})_{i_1}\cdots (q^{n_k};q^{n_k})_{i_k}}= \prod\limits_{ (a,n)\in P} (q^{a};q^n)_\infty^{r(a,n)}
\end{align}
as a Rogers-Ramanujan type identity of {\it index} $(n_1,n_2,\cdots,n_k)$. Here $t(i_1,\cdots,i_k)$ is an integer-valued function, $Q(i_1,\cdots,i_k)$ is a rational polynomial in variables $i_1,\cdots,i_k$, $n_1,\cdots, n_k$ are positive integers with $\gcd(n_1,n_2,\cdots,n_k)=1$, $S$ is a subset of $\mathbb{Z}^k$, $P$ is a finite subset of $\mathbb{Q}^2$ and $r(a,n)$ are integer-valued functions. With this notion, we see that the identities \eqref{KR-conj-1} and \eqref{KR-conj-2} are of index $(1,4,6)$ while \eqref{KR-conj-3} and \eqref{KR-conj-4} are of index $(1,2,3)$.

There are some other identities similar to \eqref{KR-conj-1}--\eqref{KR-conj-4} in the literature. First, we can find some identities involving double sums of index $(1,2)$, $(1,3)$ and $(1,4)$. For instance,  analytical forms of two conjectural partition identities of Capparelli \cite{Capparelli} were given in the work of Kanade and Russell \cite{KR-2019} as well as the work of Kur\c{s}ung\"{o}z \cite{Kursungoz}. These two identities are all of index $(1,3)$ and one of them is
\begin{align}\label{Capparelli-eq}
\sum_{i,j\geq 0}\frac{q^{2i^2+6ij+6j^2}}{(q;q)_i(q^3;q^3)_j}&=\frac{1}{(q^2,q^3,q^9,q^{10};q^{12})_\infty}.
\end{align}
Kur\c{s}ung\"{o}z \cite{Kursungoz} also found four identities of index $(1,4)$. Five conjectural identities of index $(1,3)$ were presented in \cite[Conjecture 6.1]{Kursungoz-AnnComb} such as
\begin{align}
\sum_{i,j\geq 0}\frac{q^{i^2+3j^2+3ij}}{(q;q)_i(q^3;q^3)_j}=\frac{1}{(q,q^3,q^6,q^8;q^9)_\infty}. \label{K-conj-1}
\end{align}
They are based on the work of Kanade and Russell \cite{KR-2015} and so far remain open.

Andrews \cite{Andrews2019} and Takigiku and Tsuchioka \cite{Takigiku-2019} provided some identities of index $(1,2)$, which can be proved by summing over one of the index first and then summing over the second index. Uncu and Zudilin \cite{Uncu-Zudilin} presented two identities of index $(1,2)$ and mentioned that they can be explained as instances of Bressoud's identities \cite{Bressoud1979}. Berkovich and Uncu \cite{Berkovich} proved an identity of index $(1,3)$. In 2021,  Andrews and Uncu \cite{Andrews-Uncu} proved an identity of index $(1,3)$ and further conjectured that \cite[Conjecture 1.2]{Andrews-Uncu}
\begin{align}\label{AU-conj}
\sum_{i,j\geq 0}\frac{(-1)^jq^{3j(3j+1)/2+i^2+3ij+i+j}}{(q;q)_i(q^3;q^3)_j}=\frac{1}{(q^2,q^3;q^6)_\infty}.
\end{align}
This was first proved  by Chern \cite{Chern} and then by Wang \cite{Wang}. Through the integral method, Wang \cite{Wang} also provided new proofs to some other double sum Rogers-Ramanujan type identities of indexes $(1,2)$, $(1,3)$ and $(1,4)$.

As for identities involving triple sums or quadruple sums, besides the Kanade-Russell identities of indexes $(1,2,3)$ and $(1,4,6)$ such as \eqref{KR-conj-1}--\eqref{KR-conj-4}, there are other known identities of indexes $(1,1,6)$, $(1,2,2)$, $(1,2,3)$, $(1,1,1,2)$, $(1,2,2,4)$ and $(1,2,3,4)$. For example, Rosengren \cite[Eq.\ (5.3a)]{Rosengren} proved an identity of index $(1,1,6)$. Kanade and Russell \cite{KR-2019} presented four conjectural identities of index $(1,2,3,4)$.
Takigiku and Tsuchioka \cite{Takigiku} proved some identities of indexes $(1,2,2)$ and $(1,2,2,4)$, which are related to the principal characters of the level 5 and level 7 standard modules of the affine Lie algebra $A_2^{(2)}$. For example, they proved that \cite[Theorem 1.3]{Takigiku}
\begin{align}
&\sum_{i,j,k\geq 0}\frac{q^{\binom{i}{2}+8\binom{j}{2}+10\binom{k}{2}+2ij+2ik+8jk+i+4j+5k}}{(q;q)_i(q^2;q^2)_j(q^2;q^2)_k} \nonumber \\
&=\frac{1}{(q,q^3,q^4,q^5,q^7,q^9,q^{11},q^{13},q^{15},q^{16},q^{17},q^{19};q^{20})_\infty}.
\end{align}


Recently, Mc Laughlin \cite{Laughlin} applied Rosengren's method in \cite{Rosengren} to derive some new Rogers-Ramanujan type identities including the following one of index $(1,2,3)$
\begin{align}\label{Laughlin123}
\sum_{i,j,k\geq 0} \frac{(-1)^j q^{(3k+2j-i)(3k+2j-i-1)/2+j(j-1)-i+6j+6k}}{(q;q)_i(q^2;q^2)_j(q^3;q^3)_k}=\frac{(-1;q)_\infty (q^{18};q^{18})_\infty}{(q^3;q^3)_\infty (q^9;q^{18})_\infty}.
\end{align}
Note that in \cite{Laughlin}, such identities are called as identities of Kanade-Russell type. In the way of finding generalizations of Capparelli's first partition identity, Dousse and Lovejoy \cite[Eqs.\ (2.6),(2.7)]{Dousse-Lovejoy} proved the following identity of index $(1,1,1,2)$:
\begin{align}\label{DL1112}
    \sum_{i,j,k,l\geq 0} \frac{a^{i+l}b^{j+l}q^{\binom{i+j+k+2l+1}{2}+\binom{i+1}{2}+\binom{j+1}{2}+l}}{(q;q)_i(q;q)_j(q;q)_k(q^2;q^2)_l}=(-q;q)_\infty (-aq^2,-bq^2;q^2)_\infty.
\end{align}

Motivated by the above works, in this paper, we will use the integral method to establish some Rogers-Ramanujan type identities of the following indexes
$$(1,1),(1,2), (1,1,1), (1,1,2), (1,1,3), (1,2,2), (1,2,3), (1,2,4).$$
Most of our results are new. Some of them contain additional parameters and thus indicate infinite families of Rogers-Ramanujan type identities. For instance, we prove that (see Theorems \ref{thm-11-general} and \ref{thm-R-3})
\begin{align}
\sum_{i,j\geq 0} \frac{u^{i-j}q^{\binom{i}{2}+\binom{j+1}{2}+a\binom{j-i}{2}}}{(q;q)_i(q;q)_j}&=\frac{(-uq^a,-q/u,q^{a+1};q^{a+1})_\infty}{(q;q)_\infty}, \label{intro-eq-J-3}\\
\sum_{i,j,k\geq0}\frac{(-1)^{i+j}b^{-i+j}c^{i-j+k}q^{(i^{2}+(i-j+2k)^{2}-2i+3j-2k)/2}}{(q;q)_{i}(q;q)_{j}(q^{2};q^{2})_{k}}&=\frac{(-q,bq^{2}/c;q)_{\infty}(bq,c/b;q^{2})_{\infty}}
{(b^{2}q^{2}/c;q^{2})_{\infty}}.
\end{align}
Some of the identities we discovered are quite surprising. For example, we find that for any $u\in \mathbb{C}$ (see Theorems \ref{thm-4112-3} and \ref{thm-123}),
\begin{align}\label{intro-eq-4112-3}
\sum_{i,j,k\geq0}\frac{(-1)^{i+j}u^{i+3k}q^{(i^{2}-i)/2+(i-2j+3k)^{2}/4}}{(q;q)_{i}(q^{2};q^{2})_{j}(q^{3};q^{3})_{k}}&=\frac{(u^{2};q)_{\infty}(q,-u^{2};q^{2})_{\infty}}{(-u^{6};q^{6})_{\infty}}, \\
\sum_{i,j,k\geq 0}\frac{(-1)^{(i-2j+3k)/2}u^{i+k}q^{(i^{2}-i)/2+(i-2j+3k)^{2}/4}}
{(q;q)_{i}(q^{2};q^{2})_{j}(q^{3};q^{3})_{k}}
&=\frac{(q;q^{2})_{\infty}(-u^{2};q^{3})_{\infty}}
 {(u^{2};q^{6})_{\infty}}.
\end{align}
A rough look at these  identities will let us doubt their correctness. From the expression of each identity, it is expected that the left side will be a power series in $q^{1/4}$. But it turns out that it is a power series in $q$, as the right side indicates.

The rest of this paper is organized as follows. In Section \ref{sec-pre} we collect some useful $q$-series formulas which will be used to derive our identities. In Sections \ref{sec-double} and \ref{sec-triple} we present and prove identities involving double sums and triple sums, respectively. Finally, we give some concluding remarks in Section \ref{sec-concluding} including a new proof of \eqref{DL1112} via the integral method.

\section{Preliminaries}\label{sec-pre}
Throughout this paper we will denote $\zeta_n=e^{2\pi i/n}$.

First, we need Euler's $q$-exponential identities
\begin{align}\label{Euler}
\sum_{n=0}^\infty \frac{z^n}{(q;q)_n}=\frac{1}{(z;q)_\infty}, \quad \sum_{n=0}^\infty \frac{q^{\binom{n}{2}} z^n}{(q;q)_n}=(-z;q)_\infty, \quad |z|<1.
\end{align}
These two identities are corollaries of the $q$-binomial theorem
\begin{align}\label{q-binomial}
    \sum_{n=0}^\infty \frac{(a;q)_n}{(q;q)_n}z^n=\frac{(az;q)_\infty}{(z;q)_\infty}, \quad |z|<1.
\end{align}
We also need the Jacobi triple product identity
\begin{align}\label{Jacobi}
(q,z,q/z;q)_\infty=\sum_{n=-\infty}^\infty (-1)^nq^{\binom{n}{2}}z^n.
\end{align}

We recall the basic hypergeometric series
$${}_r\phi_s\bigg(\genfrac{}{}{0pt}{} {a_1,\dots,a_r}{b_1,\dots,b_s};q,z  \bigg):=\sum_{n=0}^\infty \frac{(a_1,\dots,a_r;q)_n}{(q,b_1,\dots,b_s;q)_n}\Big((-1)^nq^{\binom{n}{2}} \Big)^{1+s-r}z^n.$$
For a series $f(z)=\sum_{n=-\infty}^\infty a(n)z^n$, we shall use $[z^n]f(z)$ to denote the coefficient of $z^n$. That is, $[z^n]f(z)=a(n)$. We recall the following simple fact
\begin{align}\label{int-constant}
\oint_K f(z) \frac{dz}{2\pi iz}=[z^0]f(z),
\end{align}
where $K$ is a positively oriented and simple closed contour around the origin. This fact will be used frequently but usually without mention.

There are two steps in using the integral method to prove Rogers-Ramanujan type identities:
\begin{itemize}
\item \textbf{Step 1.} Express the sum side as a finite sum of integrals of some infinite products.
\item \textbf{Step 2.} Evaluate each of these integrals.
\end{itemize}

The first step is quite straightforward. In the proofs of all the  Rogers-Ramanujan type identities appeared in \cite{Rosengren}, \cite{Wang} and this paper, this step will be done by the use of \eqref{Euler} and \eqref{Jacobi}.

The main difficulty lies in the second step. In the book \cite[Sections 4.9 and 4.10]{GR-book}, calculations of the integral
$$\oint_K \frac{(a_1z,\cdots,a_Az,b_1/z,\cdots,b_B/z;q)_\infty}{(c_1z,\cdots,c_Cz,d_1/z,\cdots,d_D/z;q)_\infty}z^{m}\frac{dz}{2\pi iz} $$
are given. Here $m$ is an integer,  $K$ is a deformation of the (positively oriented) unit circle so that the poles of $1/(c_1z,\cdots,c_Cz;q)_\infty$ lie outside the contour and the origin and poles of $1/(d_1/z,\cdots,d_D/z;q)_\infty$ lie inside the contour. Throughout this paper, all the integral paths will be chosen in this way and we will omit them from the integral symbol. We will not need these  general calculations. Instead, we recall some known formulas which will suffice to establish our multi-sum Rogers-Ramanujan type identities.

First, from \cite[Eq.\ (4.10.8)]{GR-book} we find that when $|a_1a_2a_3|<|c_1c_2c_3|$,
\begin{align}\label{GR41010}
&\oint \frac{(a_{1}z,a_{2}z,a_{3}z,b_{1}/z;q)_{\infty}}
{(c_{1}z,c_{2}z,c_{3}z,d_{1}/z;q)_{\infty}}\frac{dz}{2\pi iz} \\
& = \frac{(a_{1}d_{1},a_{2}d_{1},a_{3}d_{1},b_{1}/d_{1};q)_{\infty}}
 {(q,c_{1}d_{1},c_{2}d_{1},c_{3}d_{1};q)_{\infty}}
 \times{}_4\phi _3\left(
 \begin{gathered}
c_{1}d_{1},c_{2}d_{1},c_{3}d_{1},qd_{1}/b_{1}\\
 a_{1}d_{1},a_{2}d_{1},a_{3}d_{1}
 \end{gathered}
 ;q,b_{1}/d_{1}
 \right). \nonumber
\end{align}
From \cite[Eq.\ (4.11.2), (4.11.3)]{GR-book} we find
\begin{align}
\oint \frac{(cz/\beta,qz/c\alpha,c\alpha/z,q\beta/cz;q)_{\infty}}{(az,bz,\alpha/z,\beta/z;q)_{\infty}}\frac{dz}{2\pi iz}
=\frac{(ab\alpha\beta,c,q/c,c\alpha/\beta,q\beta/c\alpha;q)_{\infty}}{(a\alpha,a\beta,b\alpha,b\beta,q;q)_{\infty}}, \label{GR4112}
\end{align}
\begin{align}
&\oint \frac{(\delta z,qz/\gamma,\gamma/z,\gamma z/\alpha\beta,q\alpha\beta/\gamma z;q)_{\infty}}
{(az,bz,cz,\alpha/z,\beta/z;q)_{\infty}}\frac{dz}{2\pi iz} \nonumber \\
&= \frac{(\gamma /\alpha,q\alpha/\gamma ,\gamma/\beta,q\beta/\gamma,\delta/a,\delta/b,\delta/c;q)_{\infty}}
 {(a\alpha,a\beta,b\alpha,b\beta,c\alpha,c\beta,q;q)_{\infty}},  \label{GR4113}
\end{align}
where $\delta=abc\alpha\beta$, $abc\alpha\beta\gamma\neq 0$ and
$$a\alpha,a\beta,b\alpha,b\beta,c\alpha,c\beta \neq q^{-n}, \quad n=0,1,2,\dots.$$
Clearly, \eqref{GR4112} follows from \eqref{GR4113} after letting $c\rightarrow 0$.

Next, we recall some identities in Rosengren's work \cite{Rosengren}.  From \cite[Eq.\ (3.2)]{Rosengren} we know that when $\alpha_1\alpha_2=\beta_1\beta_2\beta_3$,
\begin{align}\label{R32}
\oint \frac{(\alpha_1z,\alpha_2z,qz,1/z;q)_\infty}{(\beta_1z,\beta_2z,\beta_3z;q)_\infty}\frac{\diff z}{2\pi iz}=\frac{(\beta_1,\alpha_1/\beta_1;q)_\infty}{(q;q)_\infty}{}_2\phi_1\bigg(\genfrac{}{}{0pt}{}{\alpha_2/\beta_2,\alpha_2/\beta_3}{\beta_1};q,\frac{\alpha_1}{\beta_1}\bigg).
\end{align}

From the proof of \cite[Proposition\ 3.2]{Rosengren}, we conclude that
\begin{align}\label{Prop32-proof}
 \oint \frac{(abz,cz,qz/t,t/z;q)_{\infty}}{(az,bz,cz/t,d/z;q)_{\infty}}\frac{dz}{2\pi iz}=\frac{(abd,dq/t,t,c;q)_{\infty}}{(q,ad,bd,cd/t;q)_{\infty}}
 {}_3\phi _2\left(
 \begin{gathered}
 a,b,cd/t\\
 c,abd
 \end{gathered}
 ;q,t
 \right).
\end{align}

Using the above formulas in Step 2, we can convert the sum-side of our Rogers-Ramanujan type identities to a ${}_r\phi_s$ series. Then to complete Step 2, it remains to evaluate this ${}_r\phi_s$ series. Here we recall the $q$-Gauss summation formula \cite[(\uppercase\expandafter{\romannumeral2}. 8)]{GR-book}
\begin{align}\label{q-Gauss}
{}_2\phi_1\bigg(\genfrac{}{}{0pt}{}{a,b}{c};q,\frac{c}{ab}  \bigg)=\frac{(c/a,c/b;q)_\infty}{(c,c/ab;q)_\infty},
\end{align}
the Bailey-Daum summation formula \cite[(\uppercase\expandafter{\romannumeral2}. 9)]{GR-book}
\begin{align}\label{BD}
{}_2\phi_1\bigg(\genfrac{}{}{0pt}{} {a,b}{aq/b};q,-\frac{q}{b}  \bigg)=\frac{(-q;q)_\infty (aq,aq^2/b^2;q^2)_\infty}{(aq/b,-q/b;q)_\infty}
\end{align}
and the $q$-Dixon summation formula \cite[(\uppercase\expandafter{\romannumeral2}.13)]{GR-book}
\begin{align}\label{II13}
{}_4\phi _3\left(
 \begin{gathered}
a,-qa^{1/2},b,c\\
 -a^{1/2},aq/b,aq/c
 \end{gathered}
 ;q, \frac{qa^{1/2}}{bc}
 \right)
 =\frac{(aq,qa^{1/2}/b,qa^{1/2}/c,aq/bc;q)_{\infty}}
 {(aq/b,aq/c,qa^{1/2},qa^{1/2}/bc;q)_{\infty}}.
\end{align}

\section{Identities involving double sums}\label{sec-double}

In this section, we present some identities involving double sums of indexes $(1,1)$ and $(1,2)$.

\subsection{Identities of index $(1,1)$}

\begin{theorem}\label{thm-R-1}
We have
\begin{align}
\sum_{i,j\geq0}\frac{(-1)^{i+j}u^{i}v^{j}q^{((i-j)^{2}-i-j)/2}}{(q;q)_{i}(q;q)_{j}}= \frac{(u,v;q)_{\infty}}{(uv/q;q)_{\infty}}. \label{eq-R-1}
\end{align}
\end{theorem}
Note that the identity \eqref{eq-R-1} is symmetric in $u$ and $v$.
\begin{proof}
Setting $a=c=0$ in \eqref{Prop32-proof}, we deduce that
\begin{align}
(q;q)_{\infty}\oint \frac{(qz/t,t/z;q)_{\infty}}{(bz,d/z;q)_{\infty}}\frac{dz}{2\pi iz}
=\frac{(dq/t,t;q)_{\infty}}{(bd;q)_{\infty}}
\sum_{n\geq0}\frac{(b;q)_{n}}{(q;q)_{n}}t^{n}
  =\frac{(dq/t,bt;q)_{\infty}}
{(bd;q)_{\infty}},
\end{align}
where for the last equality we used \eqref{q-binomial}.

Now by \eqref{Euler} and \eqref{Jacobi},
\[
\begin{split}
 LHS&=\oint  \sum_{i,j\geq0}\sum_{k= -\infty}^{\infty}\frac{(bz)^{i} (d/z)^{j} (-t/z)^{k} q^{(k^{2}-k)/2}}{(q;q)_{i}(q;q)_{j}} \frac{dz}{2\pi iz}\\
 &=\sum_{i,j\geq0}\frac{(-1)^{i+j}b^{i}d^{j}t^{i-j}q^{((i-j)^{2}-i+j)/2}}{(q;q)_{i}(q;q)_{j}}.
\end{split}
\]
Here we used \eqref{int-constant} for the second equality. This proves the desired identity after replacing $bt$ by $u$, and $dq/t$ by $v$.
\end{proof}
We can also prove Theorem \ref{thm-R-1} by the following way.
\begin{proof}[Second proof of Theorem \ref{thm-R-1}]
Summing over $i$ first using \eqref{Euler} and then applying \eqref{q-binomial}, we have
\begin{align*}
&\sum_{i,j\geq0}\frac{(-1)^{i+j}u^{i}v^{j}q^{((i-j)^{2}-i-j)/2}}{(q;q)_{i}(q;q)_{j}}=\sum_{j\geq 0} \frac{(-v)^{j}q^{(j^2-j)/2}}{(q;q)_j} \sum_{i\geq 0}\frac{(-uq^{-j})^{i}q^{(i^2-i)/2}}{(q;q)_i} \nonumber \\
&=\sum_{j\geq 0} \frac{(uq^{-j};q)_\infty (-v)^jq^{(j^2-j)/2}}{(q;q)_j} =(u;q)_\infty \sum_{j\geq 0}\frac{(uv/q)^{j}(q/u;q)_j}{(q;q)_j} \nonumber \\
&=\frac{(u,v;q)_\infty }{(uv/q;q)_\infty}. \qedhere
\end{align*}
\end{proof}

Setting $u=-q$, $v=-q^{1/2}$ and $u=-q$, $v=-q$ in Theorem \ref{thm-R-1}, we obtain
\begin{align}
  \sum_{i,j\geq 0}\frac{q^{((i-j)^{2}+i)/2}}{(q;q)_{i}(q;q)_{j}}&=\frac{1}{(q^{1/2};q)_{\infty}^{2}}, \label{eq-thm3.1-cor-1} \\
  \sum_{i,j\geq 0}\frac{q^{((i-j)^{2}+i+j)/2}}{(q;q)_{i}(q;q)_{j}}&=\frac{(q^{2};q^{2})_{\infty}^{2}}{(q;q)_{\infty}^{3}}.\label{eq-thm3.1-cor-1.1}
\end{align}

\begin{theorem}\label{thm-4112-2}
We have
\begin{equation}\label{eq-4112-2}
\sum_{i,j\geq0}\frac{(-1)^{i+j}u^{i}q^{(i-j)^{2}}}{(q^{2};q^{2})_{i}(q^{2};q^{2})_{j}} =\frac{(u;q)_{\infty}(q;q^{2})_{\infty}}{(u;q^{2})_{\infty}^{2}}.
\end{equation}
\end{theorem}
\begin{proof}
Setting $c=q^{1/2}$, $a=-b$ and $\alpha=-\beta$ in \eqref{GR4112}, then multiplying both sides by $(q^{2};q^{2})_{\infty}$, we obtain by \eqref{Euler} and \eqref{Jacobi} that the left side of \eqref{GR4112} becomes
\begin{align*}
 LHS&=(q^{2};q^{2})_{\infty}\oint \frac{(qz^{2}/\alpha^{2},q\alpha^{2}/z^{2};q^{2})_{\infty}}
{(a^{2}z^{2},\alpha^{2}/z^{2};q^{2})_{\infty}}\frac{dz}{2\pi iz}\\
 &=\oint \sum_{i,j\geq0}\sum_{k= -\infty}^{\infty}\frac{(a^{2}z^{2})^{i} (\alpha^{2}/z^{2})^{j} (-q\alpha^{2}/z^{2})^{k}q^{k^{2}-k}}{(q^{2};q^{2})_{i}(q^{2};q^{2})_{j}} \frac{dz}{2\pi iz}\\
 &=
 \sum_{i,j\geq0}\frac{(-1)^{i+j}a^{2i}\alpha^{2i}q^{(i-j)^{2}}}{(q^{2};q^{2})_{i}(q^{2};q^{2})_{j}},
\end{align*}
and the right side of \eqref{GR4112} becomes
\begin{align*}
 RHS=\frac{(a^{2}\alpha^{2};q)_{\infty}(q;q^{2})_{\infty}}{(a^{2}\alpha^{2};q^{2})_{\infty}^{2}}.
\end{align*}
This proves the theorem after replacing $\alpha^2 a^2$ by $u$.
\end{proof}

For example, if we set $u=-q$, $u=-q^{3/2}$ or $u=-q^2$ in the above theorem and replace $q$ by $q^2$ in the second assignment, we obtain
\begin{align}
  \sum_{i,j\geq0}\frac{(-1)^{j}q^{(i-j)^{2}+i}}{(q^{2};q^{2})_{i}(q^{2};q^{2})_{j}}&=\frac{(q;q^{2})_{\infty}^{2}}{(q^{2};q^{4})_{\infty}^{2}}, \\
 \sum_{i,j\geq0}\frac{(-1)^{j}q^{2(i-j)^{2}+3i}}{(q^{4};q^{4})_{i}(q^{4};q^{4})_{j}}&= \frac{(q^2,q^{10};q^{8})_{\infty}(q^{3};q^{4})_{\infty}}{(q^{5};q^{4})_{\infty}}, \\
  \sum_{i,j\geq0}\frac{(-1)^{j}q^{(i-j)^{2}+2i}}{(q^{2};q^{2})_{i}(q^{2};q^{2})_{j}}&=\frac{(q,q^{2},q^{6};q^{4})_{\infty}}{(q^{5};q^{4})_{\infty}}.
\end{align}

\begin{theorem}\label{thm-T11}
We have
\begin{align}
\sum_{i,j\geq0}\frac{(-1)^{i+j}q^{(i-j)^{2}/2}(q^{j}-q^{i+1/2})}{(q;q)_{i}(q;q)_{j}}  &=\frac{(q^{1/2};q)_{\infty}^{2}}
 {(q;q)_{\infty}}, \label{T11-2}\\
 \sum_{i,j\geq0}\frac{q^{(i-j)^{2}/2}(q^{j}+q^{i+1/2})}{(q;q)_{i}(q;q)_{j}}  &=\frac{(q;q^{2})_{\infty}}
 {(q^{2};q^{2})_{\infty}(q^{1/2};q)_{\infty}^{2}}. \label{T11-3}
\end{align}
\end{theorem}
\begin{proof}
From \eqref{GR41010} and \eqref{II13} we have
\begin{align}\label{Eq14}
&\oint \frac{(-a^{1/2}z,a^{1/2}qz,abz,b/z;q)_{\infty}}
{(az,-a^{1/2}qz,a^{1/2}z,1/z;q)_{\infty}}\frac{dz}{2\pi iz} \nonumber \\
& = \frac{(-a^{1/2},a^{1/2}q,ab,b;q)_{\infty}}
 {(q,a,-a^{1/2}q,a^{1/2};q)_{\infty}}
{}_4\phi _3\left(
 \begin{gathered}
a,-a^{1/2}q,a^{1/2},q/b\\
 -a^{1/2},a^{1/2}q,ab
 \end{gathered}
 ;q,b
 \right)    \nonumber  \\
 &=\frac{(-a^{1/2},aq,a^{1/2}b,a^{1/2}b;q)_{\infty}}
 {(a^{1/2},a,-a^{1/2}q,a^{1/2}q;q)_{\infty}}.
\end{align}
Let $a=q^{2}$ in \eqref{Eq14}. We obtain
  \begin{align}\label{Eq15}
\oint \frac{(-qz,bq^{2}z,b/z;q)_{\infty}}
{(-q^{2}z,qz,1/z;q)_{\infty}}\frac{dz}{2\pi iz}
=\frac{(-q,q^{3},bq,bq;q)_{\infty}}
 {(q,q^{2},-q^{2},q^{2};q)_{\infty}}.
\end{align}

Setting $b=q^{-1/2}$ in \eqref{Eq15} and multiplying both sides by $(q;q)_\infty$, we see that its left side becomes
\begin{align*}
&(q;q)_{\infty}
\oint \frac{(-qz,q^{3/2}z,1/q^{1/2}z;q)_{\infty}}
{(-q^{2}z,qz,1/z;q)_{\infty}}\frac{dz}{2\pi iz} \\
&=\oint (1+qz)\sum_{i,j\geq0}\frac{(qz)^{i}(1/z)^{j}}{(q;q)_{i}(q;q)_{j}}
\sum_{k= -\infty}^{\infty}(-q^{1/2}z)^{-k}q^{(k^{2}-k)/2}\frac{dz}{2\pi iz}    \\
&=\sum_{i,j\geq0}\frac{(-1)^{i+j}q^{(i-j)^{2}/2}(q^{j}-q^{i+1/2})}{(q;q)_{i}(q;q)_{j}}, \end{align*}
and its right side becomes
\begin{align*}
RHS=\frac{(-q,q^{3},q^{1/2},q^{1/2};q)_{\infty}}
 {(q^{2},-q^{2},q^{2};q)_{\infty}}
 =\frac{(q^{1/2};q)_{\infty}^{2}}
 {(q;q)_{\infty}}.
\end{align*}
This proves \eqref{T11-2}.

Similarly, setting $b=-q^{-1/2}$ in \eqref{Eq15} and applying \eqref{Euler} and \eqref{Jacobi}, we obtain \eqref{T11-3}.
\end{proof}

Note that if we set $b=-1$ in \eqref{Eq15}, then we obtain \eqref{eq-thm3.1-cor-1.1}.

\begin{rem}\label{rem-sec3}
Similar to the second proof of Theorem \ref{thm-R-1}, Theorems \ref{thm-4112-2} and \ref{thm-T11} can also be proved by summing over one of the index first. We omit these proofs.
\end{rem}

Now we present another set of Rogers-Ramanujan type identities of index $(1,1)$. These identities are proved by repeated use of the Jacobi triple product identity, and we do not need to calculate any ${}_r\phi_s$ series.
\begin{theorem}\label{thm-11-general}
We have
\begin{align}
\sum_{i,j\geq 0} \frac{u^{i-j}q^{\binom{i}{2}+\binom{j+1}{2}+a\binom{j-i}{2}}}{(q;q)_i(q;q)_j}=\frac{(-uq^a,-q/u,q^{a+1};q^{a+1})_\infty}{(q;q)_\infty}.
\end{align}
\end{theorem}
\begin{proof}
By the Jacobi triple product identity, we have
\begin{align*}
    &(q;q)_\infty (q^a;q^a)_\infty \oint (uz,q/uz;q)_\infty (z,q^a/z;q^a)_\infty \frac{dz}{2\pi iz} \nonumber \\
    &=\oint \sum_{i,j=-\infty}^\infty (-uz)^i q^{\binom{i}{2}} (-z)^jq^{a\binom{j}{2}}\frac{dz}{2\pi iz} \nonumber \\
    &=\sum_{i=-\infty}^\infty u^iq^{(a-1)i/2}q^{(a+1)i^2/2} \nonumber \\
    &=(-uq^a,-q/u,q^{a+1};q^{a+1})_\infty.
\end{align*}
By \eqref{Euler} and \eqref{Jacobi}, the left side of this identity can also be written as
\begin{align*}
    LHS&=(q;q)_\infty \oint \sum_{i,j\geq 0}\frac{(-uz)^iq^{\binom{i}{2}}}{(q;q)_i}\cdot \frac{(-q/uz)^jq^{\binom{j}{2}}}{(q;q)_j}\cdot
    \sum_{k=-\infty}^\infty (-z)^k q^{a\binom{k}{2}}\frac{dz}{2\pi iz} \nonumber \\
    &=(q;q)_\infty\sum_{i,j\geq 0}\frac{u^{i-j}q^{\binom{i}{2}+\binom{j+1}{2}+a\binom{j-i}{2}}}{(q;q)_i(q;q)_j}.
\end{align*}
This proves the desired identity.
\end{proof}

Replacing $q$ by $q^{m_1}$ and setting  $a=m_2/m_1$ and $u=\pm q^{n}$, where $m_1,m_2>0$ and $n\in \mathbb{R}$, we obtain the following corollary.
\begin{corollary}\label{cor-Jacobi-add-1}
We have
\begin{align}
&\sum_{i,j\geq 0}\frac{q^{((m_{1}+m_{2})(i^{2}+j^{2})-2m_{2}ij+(2n-m_{1}+m_{2})(i-j))/2}}{(q^{m_{1}};q^{m_{1}})_{i}(q^{m_{1}};q^{m_{1}})_{j}} \nonumber \\
&=\frac{(-q^{m_{1}-n},-q^{m_{2}+n},q^{m_{1}+m_{2}};q^{m_{1}+m_{2}})_{\infty}}
{(q^{m_{1}};q^{m_{1}})_{\infty}}, \label{eq-J-1} \\
&\sum_{i,j\geq 0}\frac{(-1)^{i+j}q^{((m_{1}+m_{2})(i^{2}+j^{2})-2m_{2}ij+(2n-m_{1}+m_{2})(i-j))/2}}{(q^{m_{1}};q^{m_{1}})_{i}(q^{m_{1}};q^{m_{1}})_{j}} \nonumber \\
&=\frac{(q^{m_{1}-n},q^{m_{2}+n},q^{m_{1}+m_{2}};q^{m_{1}+m_{2}})_{\infty}}
{(q^{m_{1}};q^{m_{1}})_{\infty}}. \label{eq-J-2}
\end{align}
\end{corollary}

As examples, if we set $(m_1,m_2,n)=(1,3,-1)$  in \eqref{eq-J-1}, we obtain
$$\sum_{i,j=0}^\infty \frac{q^{2(i^2+j^2)-3ij}}{(q;q)_i(q;q)_j}=\frac{(-q^2,-q^2,q^4;q^4)_\infty}{(q;q)_\infty}.$$
Setting $(m_1,m_2,n)$ as $(3,4,0)$, $(3,4,1)$ or $(3,4,2)$ in \eqref{eq-J-2}, we obtain
\begin{align}
\sum_{i,j\geq 0}\frac{(-1)^{i+j}q^{(7i^{2}+7j^{2}-8ij+i-j)/2}}{(q^{3};q^{3})_{i}(q^{3};q^{3})_{j}}&=\frac{(q^{3},q^{4},q^{7};q^{7})_{\infty}}{(q^{3};q^{3})_{\infty}}, \\
\sum_{i,j\geq 0}\frac{(-1)^{i+j}q^{(7i^{2}+7j^{2}-8ij+3i-3j)/2}}{(q^{3};q^{3})_{i}(q^{3};q^{3})_{j}}&=
  \frac{(q^{2},q^{5},q^{7};q^{7})_{\infty}}{(q^{3};q^{3})_{\infty}}, \\
  \sum_{i,j\geq 0}\frac{(-1)^{i+j}q^{(7i^{2}+7j^{2}-8ij+5i-5j)/2}}{(q^{3};q^{3})_{i}(q^{3};q^{3})_{j}}&= \frac{(q,q^{6},q^{7};q^{7})_{\infty}}{(q^{3};q^{3})_{\infty}}.
\end{align}

\begin{theorem}\label{thm-J-3}
We have
\begin{align}\label{eq-thm-J-3}
&\sum_{i,j\geq0}\frac{(-1)^{i+j}u^{i-j}q^{(i^{2}-i+j^{2}-j+4a(i-j)^{2})/2}}{(q;q)_{i}(q;q)_{j}} \\
&=\frac{(u^{-1}q^{2a},uq^{2a+1},q^{4a+1};q^{4a+1})_{\infty}+ (uq^{2a},u^{-1}q^{2a+1},q^{4a+1};q^{4a+1})_{\infty}}{(q;q)_{\infty}}. \nonumber
\end{align}
\end{theorem}
\begin{proof}
By the Jacobi triple product identity, we have
\begin{align*}
 &(q;q)_{\infty}(q^{a};q^{a})_{\infty}\oint  (uz^{2},1/uz^{2};q)_{\infty}(q^{a/2}z,q^{a/2}/z;q^{a})_{\infty} \frac{dz}{2\pi iz}\\
 &= \oint  (1-uz^{2}) \sum_{i,j=-\infty}^{\infty}(-1/uz^{2})^{i}q^{(i^{2}-i)/2}(-q^{a/2}z)^{j}q^{a(j^{2}-j)/2} \frac{dz}{2\pi iz}  \\
 &= \oint  \Big(\sum_{i,j=-\infty}^{\infty}(-1/uz^{2})^{i}q^{(i^{2}-i)/2}(-q^{a/2}z)^{j}q^{a(j^{2}-j)/2} \\ &\quad -uz^{2}\sum_{i,j=-\infty}^{\infty}(-1/uz^{2})^{i}q^{(i^{2}-i)/2}(-q^{a/2}z)^{j}q^{a(j^{2}-j)/2} \Big)\frac{dz}{2\pi iz} \\
 &=\sum_{i=-\infty}^{\infty} \big((-1)^{i}u^{-i}q^{((4a+1)i^{2}-i)/2}+(-1)^{i}u^{-i}q^{((4a+1)i^{2}+i)/2}\big)  \qquad \\
 &=(u^{-1}q^{2a},uq^{2a+1},q^{4a+1};q^{4a+1})_{\infty}+ (uq^{2a},u^{-1}q^{2a+1},q^{4a+1};q^{4a+1})_{\infty}.
\end{align*}
Here the third equality follows, since in the first sum, only the terms with $j=2i$ contributes to the integral, and in the second sum, only the terms with $j=2i-2$ contributes to the integral. We have also replaced $i$ by $i+1$ in the outcome of the integral of the second sum.

By \eqref{Euler} and \eqref{Jacobi}, we see that the left side of the above identity is
\begin{align*}
 LHS&=(q;q)_{\infty}\oint  \sum_{i,j\geq0}\sum_{k= -\infty}^{\infty}\frac{(-uz^{2})^{i}q^{(i^{2}-i)/2} (-1/uz^{2})^{j} q^{(j^{2}-j)/2} (-q^{a/2}/z)^{k}q^{a(k^{2}-k)/2}}{(q;q)_{i}(q;q)_{j}} \frac{dz}{2\pi iz}\\
 &=(q;q)_{\infty}
 \sum_{i,j\geq 0}\frac{(-1)^{i+j}u^{i-j}q^{(i^{2}-i+j^{2}-j+4a(i-j)^{2})/2}}{(q;q)_{i}(q;q)_{j}}.
\end{align*}
This proves the theorem.
\end{proof}

If we set $u=\pm 1$, $q^{2a}$ and $q^{2a+1}$ in Theorem \ref{thm-J-3}, we obtain the following corollary.
\begin{corollary}\label{cor-J-4}
We have
\begin{align}\label{eq-J-3}
\sum_{i,j\geq0}\frac{(-1)^{i+j}q^{(i^{2}-i+j^{2}-j+4a(i-j)^{2})/2}}{(q;q)_{i}(q;q)_{j}}&=\frac{2(q^{2a},q^{2a+1},q^{4a+1};q^{4a+1})_{\infty}}{(q;q)_{\infty}}, \\
\sum_{i,j\geq0}\frac{q^{(i^{2}-i+j^{2}-j+4a(i-j)^{2})/2}}{(q;q)_{i}(q;q)_{j}}&=\frac{2(-q^{2a},-q^{2a+1},q^{4a+1};q^{4a+1})_{\infty}}{(q;q)_{\infty}}, \\
\sum_{i,j\geq0}\frac{(-1)^{i+j}q^{2a(i-j)}q^{(i^{2}-i+j^{2}-j+4a(i-j)^{2})/2}}{(q;q)_{i}(q;q)_{j}}&=\frac{(q,q^{4a},q^{4a+1};q^{4a+1})_\infty}{(q;q)_\infty}, \\
\sum_{i,j\geq0}\frac{(-1)^{i+j}q^{(2a+1)(i-j)}q^{(i^{2}-i+j^{2}-j+4a(i-j)^{2})/2}}{(q;q)_{i}(q;q)_{j}}&=\frac{(q^{-1},q^{4a+2},q^{4a+1};q^{4a+1})_\infty}{(q;q)_\infty}.
\end{align}
\end{corollary}

Setting $a=2$ and $a=3$ in the first two identities in Corollary \ref{cor-J-4}, we obtain
\begin{align}
\sum_{i,j\geq 0}\frac{(-1)^{i+j}q^{(i^{2}-i+j^{2}-j+8(i-j)^{2})/2}}{(q;q)_{i}(q;q)_{j}}&= \frac{2(q^{4},q^{5},q^{9};q^{9})_{\infty}}{(q;q)_{\infty}}, \\
\sum_{i,j\geq 0}\frac{(-1)^{i+j}q^{(i^{2}-i+j^{2}-j+12(i-j)^{2})/2}}{(q;q)_{i}(q;q)_{j}}&=\frac{2(q^{6},q^{7},q^{13};q^{13})_{\infty}}{(q;q)_{\infty}}, \\
\sum_{i,j\geq 0}\frac{q^{(i^{2}-i+j^{2}-j+8(i-j)^{2})/2}}{(q;q)_{i}(q;q)_{j}}&= \frac{2(-q^{4},-q^{5},q^{9};q^{9})_{\infty}}{(q;q)_{\infty}}, \\
\sum_{i,j\geq 0}\frac{q^{(i^{2}-i+j^{2}-j+12(i-j)^{2})/2}}{(q;q)_{i}(q;q)_{j}}&=\frac{2(-q^{6},-q^{7},q^{13};q^{13})_{\infty}}{(q;q)_{\infty}}.
\end{align}

\subsection{Identities of index $(1,2)$}

\begin{theorem}\label{thm-R-5}
We have
\begin{align}
\sum_{i,j\geq0}\frac{(-1)^{i}u^{i+j}q^{i^2+2ij+2j^2-i-j}}{(q;q)_{i}(q^{2};q^{2})_{j}}=(u;q^{2})_{\infty}, \label{eq-R-5a} \\
\sum_{i,j\geq0}\frac{(-1)^{i} u^{i+2j}q^{i^2+2ij+2j^2-i-j}}{(q;q)_{i}(q^{2};q^{2})_{j}}=(u;q)_{\infty}. \label{eq-R-5b}
\end{align}
\end{theorem}
\begin{proof}
Setting $\alpha_{1}=\beta_{2}$ in \eqref{R32} and using \eqref{q-binomial}, we deduce that
\begin{align}\label{eq2.1}
\oint  \frac{(\beta_{1}\beta_{3}z,qz,1/z;q)_{\infty}}{(\beta_{1}z,\beta_{3}z;q)_{\infty}}\frac{dz}{2\pi iz}&=\frac{(\beta_1,\beta_2/\beta_1;q)_\infty}{(q;q)_\infty} \sum_{n=0}^\infty \frac{(\beta_1\beta_3/\beta_2;q)_n}{(q;q)_n}\left(\frac{\beta_2}{\beta_1}\right)^n \nonumber \\
&=\frac{(\beta_{1},\beta_{3};q)_{\infty}}{(q;q)_{\infty}}.
\end{align}
Setting $\beta_{1}=-\beta_{3}$ in \eqref{eq2.1}, we obtain
\begin{align}\label{L-constant}
 (q;q)_{\infty}\oint \frac{(-\beta_{1}^{2}z,qz,1/z;q)_{\infty}}{(\beta_{1}^{2}z^{2};q^{2})_{\infty}}\frac{dz}{2\pi iz}
 = (\beta_{1}^{2};q^{2})_{\infty}.
\end{align}
By \eqref{Euler} and \eqref{Jacobi}, we see that its left side is
\begin{align*}
 LHS&=\oint \sum_{i,j\geq0}\sum_{k= -\infty}^{\infty}\frac{(\beta_{1}^{2}z)^{i}q^{(i^{2}-i)/2} (\beta_{1}^{2}z^{2})^{j} (-1/z)^{k}q^{(k^{2}-k)/2} }{(q;q)_{i}(q^{2};q^{2})_{j}} \frac{dz}{2\pi iz}\\
 &=\sum_{i,j\geq 0}\frac{(-1)^{i}\beta_{1}^{2i+2j}q^{(i^{2}+(i+2j)^{2}-2i-2j)/2}}{(q;q)_{i}(q^{2};q^{2})_{j}}.
\end{align*}
This proves \eqref{eq-R-5a} after replacing $\beta_1^2$ by $u$.

Replacing $q$ by $q^{2}$ in \eqref{eq2.1} and setting $\beta_{3}=\beta_{1}q$, we obtain
\begin{align*}
 (q^{2};q^{2})_{\infty}\oint  \frac{(\beta_{1}^{2}qz,q^{2}z,1/z;q^{2})_{\infty}}{(\beta_{1}z;q)_{\infty}}\frac{dz}{2\pi iz}
 = (\beta_{1};q)_{\infty}.
\end{align*}
By \eqref{Euler} and \eqref{Jacobi}, we see that its left side is
\begin{align*}
 LHS&=\oint  \sum_{i,j\geq 0} \sum_{k= -\infty}^{\infty}\frac{(\beta_{1}z)^{i} (-\beta_{1}^{2}qz)^{j}q^{j^{2}-j} (-1/z)^{k}q^{k^{2}-k} }{(q;q)_{i}(q^{2};q^{2})_{j}} \frac{dz}{2\pi iz}\\
 &=\sum_{i,j\geq 0}\frac{(-1)^{i}\beta_{1}^{i+2j}q^{j^{2}+(i+j)^{2}-i-j}}{(q;q)_{i}(q^{2};q^{2})_{j}}.
\end{align*}
This proves \eqref{eq-R-5b} after replacing $\beta_1$ by $u$.
\end{proof}

For example, if we set $u=q$ and $q^{2}$ in \eqref{eq-R-5a}, we obtain
\begin{align}
\sum_{i,j\geq 0}\frac{(-1)^{i}q^{i^{2}+2ij+2j^2}}{(q;q)_{i}(q^{2};q^{2})_{j}}&=(q;q^{2})_{\infty}, \label{add-12-1}\\
\sum_{i,j\geq 0}\frac{(-1)^{i}q^{i^{2}+2ij+2j^2+i+j}}{(q;q)_{i}(q^{2};q^{2})_{j}}&=(q^{2};q^{2})_{\infty}. \label{add-12-2}
\end{align}
If we set $u=q$ and $-q$ in \eqref{eq-R-5b}, we obtain
\begin{align}
\sum_{i,j\geq 0}\frac{(-1)^{i}q^{i^{2}+2ij+2j^{2}+j}}{(q;q)_{i}(q^{2};q^{2})_{j}}&= (q;q)_{\infty}, \label{add-12-3} \\
 \sum_{i,j\geq 0}\frac{q^{i^{2}+2ij+2j^{2}+j}}{(q;q)_{i}(q^{2};q^{2})_{j}}&=\frac{1}{(q;q^{2})_{\infty}}. \label{add-12-4}
\end{align}
Note that \eqref{add-12-4} recovers \cite[Eq.\ (1.20)]{Wang} and hence \eqref{eq-R-5b} can be viewed as a generalization of it.

\begin{rem}
The identity \eqref{eq-R-5a} can also be deduced from the following identity in Lovejoy's work \cite[Eq.\ (1.7)]{Lovejoy2006}:
\begin{align}\label{Lovejoy-constant-eq}
    [z^0]\frac{(-azq,-zq,-1/z;q)_\infty}{(-aqz^2;q^2)_\infty}=(-aq;q^2)_\infty.
\end{align}
Indeed, after setting $aq=-\beta_1^2$ and replacing $z$ by $-z$, we see that this identity is equivalent to \eqref{L-constant}. Lovejoy \cite{Lovejoy2006} also provided a  partition interpretation to \eqref{Lovejoy-constant-eq} and hence the identity \eqref{eq-R-5a} can also be explained as a partition identity.
\end{rem}


\section{Identities involving triple sums}\label{sec-triple}

In this section, we will establish Rogers-Ramanujan type identities involving triple sums.

\subsection{Identities of index $(1,1,1)$}

\begin{theorem}\label{thm-R-4}
We have
\begin{align}\label{eq-111}
\sum_{i,j,k\geq0}\frac{(-1)^{j+k}\beta_{1}^{i+j}\beta_{3}^{i+k}q^{(i^{2}+(i+j+k)^{2}-2i-j-k)/2}}{(q;q)_{i}(q;q)_{j}(q;q)_{k}}=(\beta_{1},\beta_{3};q)_{\infty}.
\end{align}
\end{theorem}
\begin{proof}
Recall the identity \eqref{eq2.1}. By \eqref{Euler} and \eqref{Jacobi}, we see that its left side is
\begin{align*}
 LHS&=\frac{1}{(q;q)_{\infty}}\oint \sum_{i,j,k\geq0}\sum_{l= -\infty}^{\infty}\frac{(-\beta_{1}\beta_{3}z)^{i}q^{(i^{2}-i)/2} (\beta_{1}z)^{j} (\beta_{3}z)^{k} (-1/z)^{l}q^{(l^{2}-l)/2}}{(q;q)_{i}(q;q)_{j}(q;q)_{k}} \frac{dz}{2\pi iz}\\
 &=\sum_{i,j,k\geq0}\frac{(-1)^{j+k}\beta_{1}^{i+j}\beta_{3}^{i+k}q^{(i^{2}+(i+j+k)^{2}-2i-j-k)/2}}{(q;q)_{i}(q;q)_{j}(q;q)_{k}}.
\end{align*}
This proves the theorem.
\end{proof}

For example, if we set  $\beta_{1}=-q^{1/4}$, $\beta_{3}=-q^{1/2}$ and replace $q$ by $q^4$, we obtain
\begin{align}
 \sum_{i,j,k\geq0}\frac{q^{2i^{2}+2(i+j+k)^{2}-i-j}}{(q^4;q^4)_{i}(q^4;q^4)_{j}(q^4;q^4)_{k}}= \frac{(q^4;q^{8})_{\infty}}{(q;q^4)_{\infty}(q^{6};q^{8})_{\infty}}.
\end{align}

\begin{rem}\label{rem-111}
The identity \eqref{eq-111} appeared in Lovejoy's work \cite{Lovejoy2017} and therein is viewed as a generalization of a partition theorem of  Schur. See Section \ref{sec-concluding} for more discussion.
\end{rem}

\subsection{Identities of index $(1,1,2)$}


\begin{theorem}\label{thm-R-3}
We have
\begin{align}
\sum_{i,j,k\geq0}\frac{(-1)^{i+j}b^{-i+j}c^{i-j+k}q^{(i^{2}+(i-j+2k)^{2}-2i+3j-2k)/2}}{(q;q)_{i}(q;q)_{j}(q^{2};q^{2})_{k}}=\frac{(-q,bq^{2}/c;q)_{\infty}(bq,c/b;q^{2})_{\infty}}
{(b^{2}q^{2}/c;q^{2})_{\infty}}.
\end{align}
\end{theorem}
\begin{proof}
Setting $a=0,t=-c/b$ and $d=-q/c$ in \eqref{Prop32-proof}, by \eqref{BD} we have
\begin{align}
& (q;q)_{\infty}\oint \frac{(cz,-bqz/c,-c/bz;q)_{\infty}}{(b^{2}z^{2};q^{2})_{\infty}(-q/cz;q)_{\infty}}\frac{dz}{2\pi iz} \nonumber \\
& = \frac{(bq^{2}/c^{2},-c/b,c;q)_{\infty}}{(-bq/c,bq/c;q)_{\infty}}
 {}_2\phi _1\left(
 \begin{gathered}
b,bq/c\\
 c
 \end{gathered}
 ;q,-c/b
 \right) \nonumber \\
 &=\frac{(-q,bq^{2}/c^{2};q)_{\infty}(bq,c^{2}/b;q^{2})_{\infty}}
{(b^{2}q^{2}/c^{2};q^{2})_{\infty}}.
\end{align}
By \eqref{Euler} and \eqref{Jacobi}, its left side is
\begin{align*}
 LHS&=\oint \sum_{i,j,k\geq0}\sum_{l= -\infty}^{\infty}\frac{(-cz)^{i}q^{(i^{2}-i)/2} (-q/cz)^{j} (b^{2}z^{2})^{k} (c/bz)^{l}q^{(l^{2}-l)/2}}{(q;q)_{i}(q;q)_{j}(q^{2};q^{2})_{k}} \frac{dz}{2\pi iz} \\
 &=\sum_{i,j,k\geq0}\frac{(-1)^{i+j}c^{2i-2j+2k}b^{-i+j}q^{(i^{2}+(i-j+2k)^{2}-2i+3j-2k)/2}}{(q;q)_{i}(q;q)_{j}(q^{2};q^{2})_{k}}.
\end{align*}
Replacing $c^2$ by $c$, we prove the theorem.
\end{proof}
Setting $(b,c)=(q^{1/2},q^2)$, $(-q^{1/2},q^2)$ and $(q^{1/2},q)$ and replacing $q$ by $q^2$, we obtain
\begin{align}
\sum_{i,j,k\geq 0}\frac{(-1)^{i+j}q^{i^{2}+(i-j+2k)^{2}+i+2k}}{(q^2;q^2)_{i}(q^2;q^2)_{j}(q^4;q^4)_{k}}&= \frac{(q;q^2)_{\infty}(q^{3};q^{4})_{\infty}^{2}}{(q^2;q^{4})_{\infty}^{2}}, \\
 \sum_{i,j,k\geq 0}\frac{q^{i^{2}+(i-j+2k)^{2}+i+2k}}{(q^2;q^2)_{i}(q^2;q^2)_{j}(q^4;q^4)_{k}}&= \frac{(q^{6};q^{8})_{\infty}^{2}}{(q;q^2)_{\infty}(q^2;q^{4})_{\infty}(q^{3};q^{4})_{\infty}^{2}}, \\
 \sum_{i,j,k\geq 0}\frac{(-1)^{i+j}q^{i^{2}+(i-j+2k)^{2}-i+2j}}{(q^2;q^2)_{i}(q^2;q^2)_{j}(q^4;q^4)_{k}}&= \frac{(q,q^3;q^2)_{\infty}}{(q^2;q^2)_{\infty}}.
\end{align}

\begin{theorem}\label{thm-4112-1}
We have
\begin{align}\label{eq-4112-1}
\sum_{i,j,k\geq0}\frac{(-1)^{i}c^{2i-j+2k}d^{j}q^{(i^{2}+(i-j+2k)^{2}-2i+j-2k)/2}}{(q;q)_{i}(q;q)_{j}(q^{2};q^{2})_{k}}=\frac{(-d q/c;q)_{\infty}(c^{2};q^{2})_{\infty}}{(d^{2};q^{2})_{\infty}}.
\end{align}
\end{theorem}
\begin{proof}
Setting $\beta=-\alpha$ and $a=q/c\alpha$ in \eqref{GR4112}, we obtain
\begin{align*}
(q;q)_{\infty}\oint \frac{(-cz/\alpha,-q\alpha/cz,c\alpha/z;q)_{\infty}}{(bz;q)_{\infty}(\alpha^{2}/z^{2};q^{2})_{\infty}}\frac{dz}{2\pi iz}
=\frac{(-b\alpha q/c;q)_{\infty}(c^{2};q^{2})_{\infty}}{(\alpha^{2}b^{2};q^{2})_{\infty}}.
\end{align*}
By \eqref{Euler} and \eqref{Jacobi} we see that its left side is
\begin{align*}
 LHS&=\oint \sum_{i,j,k\geq 0}\sum_{l= -\infty}^{\infty}\frac{(-c\alpha/z)^{i}q^{(i^{2}-i)/2} (bz)^{j} (\alpha^{2}/z^{2})^{k} (cz/\alpha)^{l}q^{(l^{2}-l)/2}}{(q;q)_{i}(q;q)_{j}(q^{2};q^{2})_{k}} \frac{dz}{2\pi iz}\\
 &=
 \sum_{i,j,k\geq0}\frac{(-1)^{i}c^{2i-j+2k}\alpha^{j}b^{j}q^{(i^{2}+(i-j+2k)^{2}-2i+j-2k)/2}}{(q;q)_{i}(q;q)_{j}(q^{2};q^{2})_{k}}.
\end{align*}
This proves the theorem after replacing $\alpha b$ by $d$.
\end{proof}

For example, if we replace $q$ by $q^4$ and set $(c,d)=(q^2,q)$ or $(q^2,q^3)$, we obtain
\begin{align}
\sum_{i,j,k\geq0}\frac{(-1)^{i}q^{2i^{2}+2(i-j+2k)^{2}+j}}{(q^{4};q^{4})_{i}(q^{4};q^{4})_{j}(q^{8};q^{8})_{k}}&= \frac{(q^{4},q^{6};q^{8})_{\infty}}{(q^{2},q^{3},q^{7};q^{8})_{\infty}}, \\
\sum_{i,j,k\geq0}\frac{(-1)^{i}q^{2i^{2}+2(i-j+2k)^{2}+3j}}{(q^{4};q^{4})_{i}(q^{4};q^{4})_{j}(q^{8};q^{8})_{k}}&= \frac{(q^{4},q^{10};q^{8})_{\infty}}{(q^{5},q^{6},q^{9};q^{8})_{\infty}}.
\end{align}

\subsection{Identities of index $(1,1,3)$}

\begin{theorem}\label{thm-R-6}
We have
\begin{align}\label{eq-R-6}
\sum_{i,j,k\geq0}\frac{(-1)^{k}u^{2i+j+3k}q^{(i^{2}+j^{2}+(i+j+3k)^{2}-2i-2j-3k)/2}}{(q;q)_{i}(q;q)_{j}(q^{3};q^{3})_{k}}=\frac{(u^{3};q^{3})_{\infty}}{(u;q)_{\infty}}.
\end{align}
\end{theorem}
\begin{proof}
Setting $\beta_{1}=\zeta_3 u,\beta_{3}=\zeta_3^{2}u$  in \eqref{eq2.1}, we obtain
\begin{align*}
 (q;q)_{\infty}\oint \frac{(u^{2}z,uz,qz,1/z;q)_{\infty}}{(u^{3}z^{^{3}};q^{3})_{\infty}}\frac{dz}{2\pi iz}
 = \frac{(u^{3};q^{3})_{\infty}}{(u;q)_{\infty}}.
\end{align*}
By \eqref{Euler} and \eqref{Jacobi}, we see that its left side is
\begin{align*}
 LHS&=\oint \sum_{i,j,k\geq0}\sum_{l= -\infty}^{\infty}\frac{(-u^{2}z)^{i}q^{(i^{2}-i)/2}  (-uz)^{j}q^{(j^{2}-j)/2}(u^{3}z^{3})^{k} (-1/z)^{l}q^{(l^{2}-l)/2} }{(q;q)_{i}(q;q)_{j}(q^{3};q^{3})_{k}} \frac{dz}{2\pi iz}\\
 &=\sum_{i,j,k\geq0}\frac{(-1)^{k}u^{2i+j+3k}q^{(i^{2}+j^{2}+(i+j+3k)^{2}-2i-2j-3k)/2}}{(q;q)_{i}(q;q)_{j}(q^{3};q^{3})_{k}}.
\end{align*}
This proves \eqref{eq-R-6}.
\end{proof}

Setting $u=q$, $q^{1/3}$, $q^{2/3}$ or $q^{1/2}$ in \eqref{eq-R-6} and replacing $q$ by $q^2$ or $q^3$ when necessary, we obtain
\begin{align}
\sum_{i,j,k\geq 0}\frac{(-1)^{k}q^{(i^{2}+j^{2}+(i+j+3k)^{2}+2i+3k)/2}}{(q;q)_{i}(q;q)_{j}(q^{3};q^{3})_{k}}&=\frac{1}{(q,q^{2};q^{3})_{\infty}}, \\
\sum_{i,j,k\geq 0}\frac{(-1)^{k}q^{3(i^{2}+j^{2}+(i+j+3k)^{2})/2-(2i+4j+3k)/2}}{(q^3;q^3)_{i}(q^3;q^3)_{j}(q^{9};q^{9})_{k}}&=\frac{(q^3;q^{9})_{\infty}}{(q;q^3)_{\infty}}, \\
\sum_{i,j,k\geq 0}\frac{(-1)^{k}q^{3(i^{2}+j^{2}+(i+j+3k)^{2})/2+(2i-2j+3k)/2}}{(q^3;q^3)_{i}(q^3;q^3)_{j}(q^{9};q^{9})_{k}}&= \frac{(q^{6};q^{9})_{\infty}}{(q^{2};q^3)_{\infty}}, \\
\sum_{i,j,k\geq0}\frac{(-1)^{k}q^{i^{2}+j^{2}+(i+j+3k)^{2}-j}}{(q^2;q^2)_{i}(q^2;q^2)_{j}(q^{6};q^{6})_{k}}&= \frac{1}{(q,q^5;q^{6})_{\infty}}.
\end{align}

\subsection{Identities of index $(1,2,2)$}

\begin{theorem}\label{thm-122}
We have
\begin{align}
\sum_{i,j,k\geq0}\frac{(-1)^{j}q^{i+j^{2}+2j+(i+j-k)^{2}}}{(q;q)_{i}(q^{2};q^{2})_{j}(q^{2};q^{2})_{k}}
&=\frac{(q^{2};q^{2})_{\infty}(q^4;q^4)_\infty^2}
 {(q;q)_{\infty}^{2}}, \\
\sum_{i,j,k\geq0}\frac{(-1)^{j}q^{j^{2}+j+k}(q^{(i+j-k)^{2}}+q^{(i+j-k+1)^{2}})}{(q;q)_{i}(q^{2};q^{2})_{j}(q^{2};q^{2})_{k}}
&=\frac{(q^{2};q^{2})_{\infty}^7}
 {(q;q)_{\infty}^{4} (q^4;q^4)_\infty^2}.
\end{align}
\end{theorem}
\begin{proof}
Let $b=-q/a^{1/2}$ in \eqref{Eq14}. We obtain
\begin{align}
\oint \frac{(-a^{1/2}z,a^{1/2}qz,-q/a^{1/2}z;q)_{\infty}}
{(az,a^{1/2}z,1/z;q)_{\infty}}\frac{dz}{2\pi iz}
=\frac{(-a^{1/2},aq,-q,-q;q)_{\infty}}
 {(a^{1/2},a,-a^{1/2}q,a^{1/2}q;q)_{\infty}}.
\end{align}
When $a=q$, we have
\begin{align*}
(q;q)_{\infty}
\oint \frac{(-q^{1/2}z,q^{3/2}z,-q^{1/2}/z;q)_{\infty}}
{(qz,q^{1/2}z,1/z;q)_{\infty}}\frac{dz}{2\pi iz}
=\frac{(-q^{1/2},q^{2},-q,-q;q)_{\infty}}
 {(q^{1/2},-q^{3/2},q^{3/2};q)_{\infty}}.
\end{align*}
Replacing $q$ by $q^2$, simplifying the denominator of the integrand using
\begin{align}\label{eq-simplify}
(q^2z,qz;q^2)_\infty=(qz;q)_\infty\end{align}
and applying \eqref{Euler} and \eqref{Jacobi}, we obtain the first identity.

Let $b=-q^{1/2}/a^{1/2}$ in \eqref{Eq14}. We obtain
\begin{align}
&\oint \frac{(-a^{1/2}z,a^{1/2}qz,-a^{1/2}q^{1/2}z,-q^{1/2}/a^{1/2}z;q)_{\infty}}
{(az,-a^{1/2}qz,a^{1/2}z,1/z;q)_{\infty}}\frac{dz}{2\pi iz} \nonumber \\
&=\frac{(-a^{1/2},aq,-q^{1/2},-q^{1/2};q)_{\infty}}
 {(a^{1/2},a,-a^{1/2}q,a^{1/2}q;q)_{\infty}}.
\end{align}
When $a=q$, we have
\begin{align*}
(q;q)_{\infty}
\oint (1+q^{1/2}z)\frac{(q^{3/2}z,-qz,-1/z;q)_{\infty}}
{(q^{1/2}z,qz,1/z;q)_{\infty}}\frac{dz}{2\pi iz}
=\frac{(q^{2};q)_{\infty}(-q^{1/2};q)_{\infty}^{3}}
 {(q^{1/2};q)_{\infty}(q^{3};q^{2})_{\infty}}.
\end{align*}
Replacing $q$ by $q^{2}$, simplifying the denominator of the integrand using \eqref{eq-simplify} and applying \eqref{Euler} and \eqref{Jacobi}, we obtain the second identity.
\end{proof}

\subsection{Identities of index $(1,2,3)$}

\begin{theorem}\label{thm-4112-3}
We have
\begin{equation}\label{eq-4112-3}
\sum_{i,j,k\geq0}\frac{(-1)^{i+j}u^{i+3k}q^{(i^{2}-i)/2+(i-2j+3k)^{2}/4}}{(q;q)_{i}(q^{2};q^{2})_{j}(q^{3};q^{3})_{k}}=\frac{(u^{2};q)_{\infty}(q,-u^{2};q^{2})_{\infty}}{(-u^{6};q^{6})_{\infty}}.
\end{equation}
\end{theorem}
\begin{proof}
Setting $c=q^{1/2}$, replacing $\alpha$ by $\zeta_2\alpha$, setting $\beta=-\zeta_2 \alpha$, $a=d\zeta_3,b=d\zeta_3^{2}$  in \eqref{GR4112}, and then multiplying both sides by $(q^{2};q^{2})_{\infty}$, we see that the left side of \eqref{GR4112} becomes
\begin{align}
 LHS&=(q^{2};q^{2})_{\infty}\oint \frac{(-qz^{2}/\alpha^{2},-q\alpha^{2}/z^{2};q^{2})_{\infty}(dz;q)_{\infty}}
{(d^{3}z^{3};q^{3})_{\infty}(-\alpha^{2}/z^{2};q^{2})_{\infty}}\frac{dz}{2\pi iz} \nonumber \\
 &=\oint \sum_{i,j,k\geq0}\sum_{l= -\infty}^{\infty}\frac{(-dz)^{i}q^{(i^{2}-i)/2} (-\alpha^{2}/z^{2})^{j} (d^{3}z^{3})^{k} (q\alpha^{2}/z^{2})^{l}q^{l^{2}-l}}{(q;q)_{i}(q^{2};q^{2})_{j}q^{3};q^{3})_{k}} \frac{dz}{2\pi iz} \nonumber \\
 &=
 \sum_{i,j,k\geq0}\frac{(-1)^{i+j}\alpha^{i+3k}d^{i+3k}q^{(i^{2}-i)/2+(i-2j+3k)^{2}/4}}{(q;q)_{i}(q^{2};q^{2})_{j}q^{3};q^{3})_{k}}-S, \qquad   \label{eq-S}
\end{align}
where
\begin{align*}
S=& \oint  \sum_{i,j,k\geq 0}\sum_{m= -\infty}^{\infty}\frac{(-dz)^{i}q^{(i^{2}-i)/2} (-\alpha^{2}/z^{2})^{j} (d^{3}z^{3})^{k} }{(q;q)_{i}(q^{2};q^{2})_{j}q^{3};q^{3})_{k}} \\
& \times (q\alpha^{2}/z^{2})^{(2m+1)/2}q^{(2m+1)^{2}/4-(2m+1)/2} \frac{dz}{2\pi iz}
\end{align*}
corresponds to the case when $l=(i-2j+3k)/2$ is not an integer, i.e., $l=(2m+1)/2$ with $m\in \mathbb{Z}$. Now we convert the integrand in the expression of $S$ back to infinite products. We have
\begin{align*}
S&=\alpha q^{1/4}(q^2;q^2)_\infty \oint \frac{(-q^2\alpha^2 /z^{2},-z^2/\alpha^{2};q^{2})_{\infty}(dz;q)_{\infty}}
{(d^{3}z^{3};q^{3})_{\infty}(-\alpha^{2}/z^{2};q^{2})_{\infty}}\frac{dz}{2\pi iz} \\
&=\alpha q^{1/4}(q^2;q^2)_\infty \oint \frac{(-z^2/\alpha^{2};q^{2})_{\infty}(dz;q)_{\infty}}
{(d^{3}z^{3};q^{3})_{\infty}(1+\alpha^2/z^2)}z^{-1}\frac{dz}{2\pi iz} \\
&=\alpha^{-1} q^{1/4}(q^2;q^2)_\infty \oint \frac{(-z^2q^2/\alpha^2;q^2)_\infty (dz;q)_\infty}{(d^3z^3;q^3)_\infty} z \frac{dz}{2\pi iz} \\
&=0.
\end{align*}
Here the last equality follows since
$$[z^0] \frac{(-z^2q^2/\alpha^2;q^2)_\infty (dz;q)_\infty}{(d^3z^3;q^3)_\infty} z=0.$$
Note that the right side of \eqref{GR4112} (after multiplication by $(q^{2};q^{2})_{\infty}$) is
\begin{align}\label{eq-S-2}
 RHS=\frac{(d^{2}\alpha^{2};q)_{\infty}(q,-d^{2}\alpha^{2};q^{2})_{\infty}}{(-d^{6}\alpha^{6};q^{6})_{\infty}}.
\end{align}
Combining \eqref{eq-S} and \eqref{eq-S-2}, replacing $d\alpha$ by $u$, we obtain the desired identity.
\end{proof}

If we set $u$ as $q^{1/2}$ or $q$ in Theorem \ref{thm-4112-3}, we obtain
\begin{align}
  \sum_{i,j,k\geq0}\frac{(-1)^{i+j}q^{(i^{2}+3k)/2+(i-2j+3k)^{2}/4}}{(q;q)_{i}(q^{2};q^{2})_{j}(q^{3};q^{3})_{k}}&=(q;q)_{\infty}(q^{3};q^{6})_{\infty}(q^{2},q^{10};q^{12})_{\infty}, \\
 \sum_{i,j,k\geq0}\frac{(-1)^{i+j}q^{(i^{2}+i+6k)/2+(i-2j+3k)^{2}/4}}{(q;q)_{i}(q^{2};q^{2})_{j}(q^{3};q^{3})_{k}}&= \frac{(q^{2};q)_{\infty}(q;q^{2})_{\infty}}{(q^{2},q^{10};q^{12})_{\infty}}.
\end{align}

\begin{theorem}\label{thm-123}
We have
\begin{align}
\sum_{i,j,k\geq 0}\frac{(-1)^{(i-2j+3k)/2}u^{i+k}q^{(i^{2}-i)/2+(i-2j+3k)^{2}/4}}
{(q;q)_{i}(q^{2};q^{2})_{j}(q^{3};q^{3})_{k}}
=\frac{(q;q^{2})_{\infty}(-u^{2};q^{3})_{\infty}}
 {(u^{2};q^{6})_{\infty}}.
\end{align}
\end{theorem}
\begin{proof}
Setting $b=\zeta_3 a,c=\zeta_3^{2}a,\alpha=-\beta$, $\gamma=q^{1/2}\alpha$ and $\delta=-a^3\alpha^2$  in \eqref{GR4113}, after multiplying both sides by $(q^2;q^2)_\infty$, we see that its left side is
\begin{align}
LHS=&(q^{2};q^{2})_{\infty}\oint \frac{(-a^{3}\alpha^{2}z;q)_{\infty}
(qz^{2}/\alpha^{2},q\alpha^{2}/z^{2};q^{2})_{\infty}}
{(a^{3}z^{3};q^{3})_{\infty}(\alpha^{2}/z^{2};q^{2})_{\infty}}\frac{dz}{2\pi iz}  \nonumber \\
&=\oint \sum_{i,j,k\geq 0}\frac{(a^{3}\alpha^{2}z)^{i}q^{(i^2-i)/2}(\alpha^{2}/z^{2})^{j}(a^{3}z^{3})^{k}}{(q;q)_{i}(q^{2};q^{2})_{j}(q^{3};q^{3})_{k}}
\sum_{l= -\infty}^{\infty}(-q\alpha^{2}/z^{2})^{l}q^{l^{2}-l}\frac{dz}{2\pi iz}  \nonumber  \\
&=\sum_{i,j,k\geq 0}\frac{(-1)^{(i-2j+3k)/2}a^{3i+3k}\alpha^{3i+3k}q^{(i^{2}-i)/2+(i-2j+3k)^{2}/4}}
{(q;q)_{i}(q^{2};q^{2})_{j}(q^{3};q^{3})_{k}}-S, \label{add-S}
\end{align}
where
\begin{align*}
S&=\oint \sum_{i,j,k\geq 0}\frac{(a^{3}\alpha^{2}z)^{i}q^{(i^2-i)/2}(\alpha^{2}/z^{2})^{j}(a^{3}z^{3})^{k}}{(q;q)_{i}(q^{2};q^{2})_{j}(q^{3};q^{3})_{k}} \nonumber \\
&\quad \times \sum_{m= -\infty}^{\infty}(-q\alpha^{2}/z^{2})^{(2m+1)/2}q^{(2m+1)^{2}/4-(2m+1)/2}\frac{dz}{2\pi iz}  \\
&=\zeta_2\alpha q^{1/4}(q^2;q^2)_\infty
\oint\frac{(-a^{3}\alpha^{2}z;q)_{\infty}
(z^{2}/\alpha^{2},q^{2}\alpha^{2}/z^{2};q^{2})_{\infty}}
{(a^{3}z^{3};q^{3})_{\infty}(\alpha^{2}/z^{2};q^{2})_{\infty}} z^{-1}\frac{dz}{2\pi iz}  \\
&=-\zeta_2\alpha^{-1} q^{1/4}
\oint \frac{(-a^{3}\alpha^{2}z;q)_{\infty}
(q^{2}z^{2}/\alpha^{2};q^{2})_{\infty}}
{(a^{3}z^{3};q^{3})_{\infty}}z\frac{dz}{2\pi iz}  \\
&=0.
\end{align*}
The right side of \eqref{GR4113} (after multiplication by $(q^2;q^2)_\infty$) is
\begin{align}
RHS=\frac{(q;q^{2})_{\infty}(-a^{6}\alpha^{6};q^{3})_{\infty}}
 {(a^{6}\alpha^{6};q^{6})_{\infty}}. \label{add-S-2}
\end{align}
Combining \eqref{add-S} and \eqref{add-S-2}, and replacing $a^3\alpha^3$ by $u$, we obtain the desired identity.
\end{proof}

Setting $u$ as $-\zeta_2 q^{3/2}$  and $q^{3/2}$  in Theorem \ref{thm-123}, we obtain
\begin{align}
\sum_{i,j,k\geq0}\frac{(-1)^{i+j}q^{(i^{2}+2i+3k)/2+(i-2j+3k)^{2}/4}}
{(q;q)_{i}(q^{2};q^{2})_{j}(q^{3};q^{3})_{k}}
&=(q;q^{2})_{\infty}(q^{3};q^{6})_{\infty}^{2}(q^{12};q^{12})_{\infty}, \\
\sum_{i,j,k\geq0}\frac{(-1)^{(i-2j+3k)/2}q^{(i^{2}+2i+3k)/2+(i-2j+3k)^{2}/4}}
{(q;q)_{i}(q^{2};q^{2})_{j}(q^{3};q^{3})_{k}}
&=\frac{(q,q^{5};q^{6})_{\infty}}{(q^{3};q^{6})_{\infty}}.
\end{align}

 \subsection{Identities of index $(1,2,4)$}
\begin{theorem}
We have
\begin{align}
\sum_{i,j,k\geq 0} \frac{(-1)^{k}q^{(i+j+2k)(i+j+2k-1)+j+2k^2}u^{i+2j+4k}}{(q;q)_{i}(q^2;q^2)_{j}(q^4;q^4)_{k}}&=(-u;q)_\infty, \\
\sum_{i,j,k\geq 0}\frac{(-1)^{j}q^{(i+j+2k)(i+j+2k-1)+2i+3j+2k^2+6k}}{(q;q)_{i}(q^2;q^2)_{j}(q^4;q^4)_{k}}&=\frac{(q^4,q^{12},q^{16};q^{16})_\infty}{(q^2;q^2)_\infty}, \\
\sum_{i,j,k\geq 0}\frac{(-1)^{j}q^{(i+j+2k)^2+2k^2}}{(q;q)_{i}(q^2;q^2)_{j}(q^4;q^4)_{k}}&=\frac{(q^8;q^8)_\infty^2}{(q^2;q^2)_\infty (q^{16};q^{16})_\infty}.
\end{align}
\end{theorem}

\begin{proof}
Let
\begin{align}
&H(u,v,w)=H(u,v,w;q) \nonumber \\
&:=\sum_{i,j,k\geq 0}\frac{u^{i}v^{j}(-w)^{k}q^{(i+j+2k)(i+j+2k-1)+2k(k-1)}}{(q;q)_{i}(q^2;q^2)_{j}(q^4;q^4)_{k}}.
\end{align}
We have by \eqref{int-constant} that
\begin{align*}
H(u,v,w)=\oint\sum_{i=0}^\infty \frac{(uz)^{i}}{(q;q)_{i}}\sum_{j=0}^\infty \frac{(vz)^{j}}{(q^2;q^2)_{j}}\sum_{k=0}^\infty \frac{(-wz^2)^{k}q^{2k(k-1)}}{(q^4;q^4)_{k}}\sum_{l=-\infty}^\infty (1/z)^{l}q^{l(l-1)}  \frac{\diff z}{2\pi iz}.
\end{align*}
Hence by \eqref{Euler} and \eqref{Jacobi},
\begin{align*}
H(u,v,w)&=(q^2;q^2)_\infty \oint \frac{(wz^2;q^4)_\infty (-1/z,-q^2z;q^2)_\infty}{(uz;q)_\infty (vz;q^2)_\infty} \frac{\diff z}{2\pi iz} \\
&=(q^2;q^2)_\infty \oint \frac{(w^{1/2}z,-w^{1/2}z,-1/z,-q^2z;q^2)_\infty}{(uz,uzq,vz;q^2)_\infty} \frac{\diff z}{2\pi iz} \\
&=(q^2;q^2)_\infty \oint \frac{(w^{1/2}z,-w^{1/2}z,1/z,q^2z;q^2)_\infty}{(-uz,-uzq,-vz;q^2)_\infty} \frac{\diff z}{2\pi iz}.
\end{align*}
Here for the last line we have replaced $z$ by $-z$. When $w=u^2vq$, we can apply \eqref{R32} with $(\alpha_1,\alpha_2,\beta_1,\beta_2,\beta_3)=(w^{1/2},-w^{1/2},-v,-u,-uq)$ to deduce that
\begin{align}\label{124F-exp}
H(u,v,w;q)=(-v,-w^{1/2}/v;q^2)_\infty \cdot {}_2\phi_1 \bigg(\genfrac{}{}{0pt}{}{w^{1/2}/u,w^{1/2}/uq}{-v};q^2,-w^{1/2}/v \bigg).
\end{align}

We now specialize the choices of $(u,v,w)$ so that the ${}_2\phi_1$ series becomes a nice infinite product.

By \eqref{124F-exp} and the $q$-Gauss summation \eqref{q-Gauss} ,we have
\begin{align*}
H(u,u^2q,u^4q^2)&=(-u^2q,-1;q^2)_\infty \cdot {}_2\phi_1 \bigg(\genfrac{}{}{0pt}{}{uq,u}{-u^2q};q^2,-1  \bigg) \\
&=(-u^2q,-1;q^2)_\infty \cdot \frac{(-u,-uq;q^2)_\infty}{(-u^2q,-1;q^2)_\infty} \\
&=(-u;q)_\infty.
\end{align*}
By \eqref{124F-exp} and  the Bailey-Daum summation \eqref{BD}, we get
\begin{align*}
H(q^2,-q^3,-q^8)&=(q^3,\zeta_2 q;q^2)_\infty \cdot {}_2\phi_1 \bigg(\genfrac{}{}{0pt}{}{\zeta_2 q^2,\zeta_2q}{q^3};q^2,\zeta_2 q  \bigg) \\
&=(q^3,\zeta_2 q;q^2)_\infty \cdot \frac{(-q^2;q^2)_\infty (\zeta_2q^4,-\zeta_2q^4;q^4)_\infty}{(q^3,\zeta_2q;q^2)_\infty} \\
&=\frac{(q^4,q^{12},q^{16};q^{16})_\infty}{(q^2;q^2)_\infty}.
\end{align*}
Similarly, by \eqref{124F-exp} and \eqref{BD} we get
\begin{align*}
H(q,-q,-q^4)&=(q,\zeta_2q;q^2)_\infty  \cdot {}_2\phi_1\bigg(\genfrac{}{}{0pt}{}{\zeta_2,\zeta_2q}{q};q^2,\zeta_2q  \bigg) \\
&=(q,\zeta_2q;q^2)_\infty \cdot \frac{(-q^2;q^2)_\infty (\zeta_2q^2,-\zeta_2q^2;q^4)_\infty}{(q,\zeta_2q;q^2)_\infty} \\
&=\frac{(q^8;q^8)_\infty^2}{(q^2;q^2)_\infty(q^{16};q^{16})_\infty}. \qedhere
\end{align*}
\end{proof}

\section{Concluding Remarks}\label{sec-concluding}

We give several remarks before closing this paper.

First, though here we only discuss identities involving double sums and triple sums, the integral method can also be applied to deduce identities with more summation indexes. For example, we can use the integral method to give a new proof to the identity \eqref{DL1112}, which is of index $(1,1,1,2)$.
\begin{proof}[Proof of \eqref{DL1112}]
Using (\ref{Euler}), (\ref{Jacobi}) and (\ref{int-constant}), we have
\begin{align} \label{Lovejoy1}
   & \sum_{i,j,k,l\geq 0} \frac{a^{i+l}b^{j+l}q^{\binom{i+j+k+2l+1}{2}+\binom{i+1}{2}+\binom{j+1}{2}+l}}{(q;q)_i(q;q)_j(q;q)_k(q^2;q^2)_l}   \nonumber\\
    &=(q;q)_\infty \oint \frac{(-aqz,-bqz,-z,-q/z;q)_{\infty}}
{(z;q)_{\infty}(abqz^{2};q^{2})_{\infty}}\frac{dz}{2\pi iz}   \nonumber\\
&=(q;q)_\infty \oint \frac{(-aqz,-bqz,-z,-q/z;q)_{\infty}}
{(z,(abq)^{1/2}z,-(abq)^{1/2}z;q)_{\infty}}\frac{dz}{2\pi iz}.
 \end{align}

 Using (\ref{q-binomial}), we have
 \begin{align} \label{bi1}
 \frac{(-bdq^{n+1};q)_\infty }{(dq^{n};q)_\infty }
 =\sum_{m\geq 0} \frac{(-bq;q)_m}{(q;q)_m}(dq^{n})^{m}.
 \end{align}

 Using Heine's transformations of ${}_2\phi_1$ series \cite[(\uppercase\expandafter{\romannumeral3}. 3)]{GR-book}
 \begin{align}
 {}_2\phi_1\bigg(\genfrac{}{}{0pt}{}{a,b}{c};q,z  \bigg)=\frac{(abz/c;q)_\infty}{(z;q)_\infty}{}_2\phi_1\bigg(\genfrac{}{}{0pt}{}{c/a,c/b}{c};q,abz/c  \bigg),
 \end{align}
we deduce that
\begin{align} \label{he1}
&\sum_{n\geq 0} \frac{((abq)^{1/2}d,-(abq)^{1/2}d;q)_n}{(q,-adq;q)_n}(-q^{m+1}/d)^{n}  \nonumber \\
&=\frac{(-bq^{m+1};q)_\infty }{(-q^{m+1}/d;q)_\infty }\sum_{n\geq 0} \frac{(-(aq/b)^{1/2},(aq/b)^{1/2};q)_n}{(q,-adq;q)_n}(-bq^{m+1})^{n}.
\end{align}

By (\ref{GR41010}), we have
\begin{align} \label{Lovejoy2}
&(q;q)_\infty\oint \frac{(-aqz,-bqz,-z,-q/z;q)_{\infty}}
{(z,(abq)^{1/2}z,-(abq)^{1/2}z,d/z;q)_{\infty}}\frac{dz}{2\pi iz} \nonumber \\
&=\frac{(-adq,-bdq,-d,-q/d;q)_{\infty}}
{(d,(abq)^{1/2}d,-(abq)^{1/2}d;q)_{\infty}}
\sum_{n\geq 0} \frac{(d,(abq)^{1/2}d,-(abq)^{1/2}d;q)_n}{(q,-adq,-bdq;q)_n}(-q/d)^{n}  \nonumber \\
&=\frac{(-adq,-d,-q/d;q)_{\infty}}
{((abq)^{1/2}d,-(abq)^{1/2}d;q)_{\infty}}
\sum_{n\geq 0} \frac{((abq)^{1/2}d,-(abq)^{1/2}d;q)_n}{(q,-adq;q)_n}(-q/d)^{n}
\frac{(-bdq^{n+1};q)_{\infty}}
{(dq^{n};q)_{\infty}}  \nonumber \\
&=\frac{(-adq,-d,-q/d;q)_{\infty}}
{((abq)^{1/2}d,-(abq)^{1/2}d;q)_{\infty}}
\sum_{m\geq 0} \frac{(-bq;q)_m}{(q;q)_m}d^{m} \nonumber \\
&\quad \quad \quad \quad \quad \times \sum_{n\geq 0} \frac{((abq)^{1/2}d,-(abq)^{1/2}d;q)_n}{(q,-adq;q)_n}(-q^{m+1}/d)^{n} \quad \text{(by  \ref{bi1})} \nonumber \\
&= \frac{(-adq,-d,-q/d;q)_{\infty}}
{((abq)^{1/2}d,-(abq)^{1/2}d;q)_{\infty}}
\sum_{n\geq 0} \frac{(-(aq/b)^{1/2},(aq/b)^{1/2};q)_n}{(q,-adq;q)_n}(-bq)^{n}   \nonumber \\
&\quad \quad \quad \times \sum_{m\geq 0} \frac{(-bq;q)_m}{(q;q)_m}(dq^{n})^{m}
\frac{(-bq^{m+1};q)_\infty }{(-q^{m+1}/d;q)_\infty } \quad \text{(by  \ref{he1})} \nonumber \\
&=\frac{(-adq,-d,-bq;q)_{\infty}}
{((abq)^{1/2}d,-(abq)^{1/2}d;q)_{\infty}}
\sum_{n\geq 0} \frac{(-(aq/b)^{1/2},(aq/b)^{1/2};q)_n}{(q,-adq;q)_n}(-bq)^{n} \nonumber \\
&\qquad \qquad \qquad \times \sum_{m\geq 0} \frac{(-q/d;q)_m}{(q;q)_m}(dq^{n})^{m} \nonumber \\
&=\frac{(-adq,-d,-bq;q)_{\infty}}
{((abq)^{1/2}d,-(abq)^{1/2}d;q)_{\infty}}
\sum_{n\geq 0} \frac{(aq/b;q^{2})_n}{(q,-adq;q)_n}(-bq)^{n}
\frac{(-q^{n+1};q)_\infty }{(dq^{n};q)_\infty } \nonumber \\
&=\frac{(-adq,-d,-bq,-q;q)_{\infty}}
{((abq)^{1/2}d,-(abq)^{1/2}d;q)_{\infty}}
\sum_{n\geq 0} \frac{(aq/b;q^{2})_n}{(q^{2};q^{2})_n(-adq;q)_n}(-bq)^{n}
\frac{1}{(dq^{n};q)_\infty } .
\end{align}

Letting $d\rightarrow 0$ on both sides of (\ref{Lovejoy2}), we have
\begin{align}\label{Lovejoy3}
&(q;q)_\infty \oint \frac{(-aqz,-bqz,-z,-q/z;q)_{\infty}}
{(z,(abq)^{1/2}z,-(abq)^{1/2}z;q)_{\infty}}\frac{dz}{2\pi iz}
\nonumber \\
&=(-bq,-q;q)_{\infty}\sum_{n\geq 0} \frac{(aq/b;q^{2})_n}{(q^{2};q^{2})_n}(-bq)^{n}\nonumber \\
&= \frac{(-bq,-q;q)_{\infty}(-aq^{2};q^{2})_{\infty}}{(-bq;q^{2})_{\infty}} \nonumber \\
&=(-q;q)_{\infty}(-aq^{2},-bq^{2};q^{2})_{\infty}.
\end{align}
Combining (\ref{Lovejoy1}) and (\ref{Lovejoy3}), we obtain (\ref{DL1112}).
\end{proof}

Second, there might exist partition or Lie theoretic interpretations for the identities we proved. This deserves further investigation. As promised in Remark \ref{rem-111}, we give below a brief discussion on partition interpretations of the identity \eqref{eq-111}.  Here we follow closely the lines in \cite{Lovejoy2017}.

Recall that a partition $\pi$ of $n$ is a nonincreasing sequence $\pi=(\lambda_1,\lambda_2,\dots,\lambda_s)$ of positive integers which sum up to $n$, i.e.,
$$n=\lambda_1+\lambda_2+\cdots+\lambda_s, \quad \lambda_1\geq \lambda_2\geq \cdots \geq \lambda_s\geq 1.$$
Let $T(u,v,n)$ be the number of bipartitions $(\pi_1,\pi_2)$ of $n$ such that the partition $\pi_1$ (resp.\ $\pi_2$) consists of $u$ (resp.\ $v$) distinct parts, respectively.  Then clearly, we have
\begin{align}\label{T-gen}
\sum_{n=0}^\infty T(u,v,n)a^ub^vq^n=(-aq,-bq;q)_\infty.
\end{align}
To generalize and refine Schur's partition theorem, Alladi and Gordon \cite{Alladi-Gordon-1993,Alladi-Gordon-1995} introduced a new kind of three colored partitions. As in \cite{Lovejoy2017},
we color the positive integers by three colors $a,b$ and $ab$ with the order that
$$ab<a<b.$$
Now the integers are ordered as
$$1_{ab}<1_a<1_b<2_{ab}<2_a<2_b<\cdots.$$
Let $S(u,v,n)$ be the number of three-colored partitions of $n$ with no parts $1_{ab}$, $u$ parts colored $a$ or $ab$, $v$ parts colored $b$ or $ab$, and satisfying the difference conditions described in the matrix
$$A=\bordermatrix{%
& a & b & ab  \cr
a & 1 & 2 &1 \cr
b & 1 & 1 &1 \cr
ab & 2 & 2 & 2
}.$$
Here the entry $(x,y)$ gives the minimal difference between the parts $\lambda_i$ of color $x$ and $\lambda_{i+1}$ of color $y$.  Alladi and Gordon proved that
\begin{align}\label{Alladi-Gordon-eq}
    \sum_{u,v,n\geq 0} S(u,v,n)a^ub^vq^n=(-aq,-bq;q)_\infty.
\end{align}

Through combinatorial arguments, Lovejoy \cite[Eq.\ (2.3)]{Lovejoy2017} proved that
\begin{align}\label{eq-Lovejoy-gen}
    \sum_{u,v,n\geq 0} S(u,v,n)a^ub^vq^n=\sum_{i,j,k\geq 0}\frac{a^iq^i}{(q;q)_i} \frac{b^jq^j}{(q;q)_j}\frac{(ab)^kq^kq^{\binom{k+1}{2}}}{(q;q)_k}q^{\binom{i+j+k}{2}}.
\end{align}
After simplifying the sums by the $q$-Chu-Vandermonde summation, Euler's identities \eqref{Euler} and the $q$-binomial identity \eqref{q-binomial}, Lovejoy obtained \eqref{Alladi-Gordon-eq}. Clearly, combining \eqref{Alladi-Gordon-eq} and \eqref{eq-Lovejoy-gen}, we get \eqref{eq-111}. Together with \eqref{T-gen}, we see that \eqref{eq-111} is equivalent to the partition identity
$$S(u,v,n)=T(u,v,n).$$
Once we convert our identities to partition identities like the above one, it will be quite interesting to find bijective proofs for them.

Finally, we want to emphasize on the advantage of the integral method. It allows us to prove the identities in this paper in a uniform manner. Of course, it is possible to give different proofs to our theorems. As discussed in Remarks \ref{rem-sec3}--\ref{rem-111}, one may prove some of the theorems using approaches such as summing over one of the index first. Compared with the other methods, the integral method has the advantage that it tells us how the identities are constructed and the calculations involved are streamlined.

\subsection*{Acknowledgements}
We thank Jeremy Lovejoy for some valuable comments, especially for bringing the works \cite{Dousse-Lovejoy,Lovejoy2006,Lovejoy2017} to our attention. We are also grateful to Chuanan Wei for helpful comments on the presentation of Corollaries \ref{cor-Jacobi-add-1} and \ref{cor-J-4}.
This work was supported by the National Natural Science Foundation of China (12171375).

\end{document}